\documentclass[12pt]{amsart}
\usepackage{amssymb,amsmath,amsthm,mathrsfs,verbatim,url, graphicx}
\usepackage{xcolor}
\newtheorem{theorem}{Theorem}

\newtheorem{conjecture}{Conjecture}
\newtheorem{proposition}[theorem]{Proposition}

\theoremstyle{remark}

\newtheorem*{remark}{Remark}

\numberwithin{theorem}{section} \numberwithin{equation}{section}

\newcommand{\Q}{\mathbb{Q}}

\newcommand{\Z}{\mathbb{Z}}
\newcommand{\N}{\mathbb{N}}

\newcommand{\F}{\mathbb{F}_p}

\newcommand{\X}{\mathcal{X}}

\author{Matija Kazalicki}
\address{ Department of Mathematics, University of Zagreb, Bijeni\v{c}ka cesta 30, 10000 Zagreb, Croatia}
\email{matija.kazalicki@math.hr}
\thanks{MK was supported by the QuantiXLie Center of Excellence}
\author{Daniel Kohen}
\address{Departamento de Matem\'atica, Facultad de Ciencias Exactas y Naturales, Universidad de Buenos Aires and IMAS, CONICET, Argentina}
\email{dkohen@dm.uba.ar}
\thanks{DK was partially supported by a CONICET doctoral fellowship}
\title{On a special case of Watkins' conjecture}
\begin{document}

\begin{abstract}
Watkins' conjecture asserts that  for a rational elliptic curve $E$ the degree of the modular parametrization is divisible by
$2^r$, where $r$ is the rank of $E$. In this paper we prove that if the modular degree is odd then $E$ has rank $0$. 
Moreover, we prove that the conjecture holds for all rank two rational elliptic curves of prime conductor and positive discriminant.
\end{abstract}

\maketitle

\section{Introduction and preliminaries}
Given a rational elliptic curve $E$ of conductor $N$, by the modularity theorem, there exists a morphism of a minimal degree
 \[ \phi: X_0(N) \rightarrow E ,\]
 that is defined over $\Q$, where $X_{0}(N)$ is the classical modular curve. Its degree, denoted by $m_E$, is called the \emph{modular degree}. 
While analyzing experimental data, Watkins conjectured that
for an elliptic curve of rank $r$,  $m_E$ is divisible by $2^r$ 
\cite[Conjecture ~4.1]{Wat2}. In particular, if the modular 
degree is odd, the rank should be zero; the proof of this assertion is the main 
result of this work.

The study of elliptic curves with odd modular degree was first developed in 
\cite{CE} by Calegari and Emerton, where they showed that a rational elliptic 
curve
with odd modular degree has to satisfy a series of very restrictive hypotheses. 
For a detailed list of conditions see \cite[Theorem ~1.1]{CE}.
Later, building on this work,  Yazdani \cite{Yaz} studied abelian varieties 
having odd modular degree. As a by-product of his work, he proves that if a 
rational elliptic curve has odd modular degree then it has rank $0$, except 
perhaps if it has prime conductor and even analytic rank (see 
\cite[Theorem ~3.8]{Yaz}  for a more general statement). 
The main result of this paper is the following theorem:
\begin{theorem} \label{thm:1}
If $E/\Q$ is an elliptic curve of odd modular degree, then $E$ has rank $0$.
\end{theorem}
By the aforementioned results it is enough to restrict ourselves to the case where $E$ has prime conductor $p$ and even analytic rank. Moreover, it is clear that we can assume that the curve $E$ is the strong Weil curve, that is, the kernel of the  map $J_0(p) \rightarrow E$ is connected ($J_0(p)$ is the Jacobian of $X_0(p)$).

The elliptic curve $E$ gives rise to a normalized newform $f_E \in S_{2}(\Gamma_0(p))$ by the modularity theorem. The main idea of the article is to associate to $f_E$ (or $E$) an element $v_E$ of the Picard group $\mathcal{X}$ of a certain curve $X$ (which is a disjoint union of curves of genus zero)  as in \cite{Gross}.
More precisely, $\mathcal{X}$ can be described as the free $\Z$-module of divisors supported on the isomorphism classes of supersingular elliptic curves over $\overline{\F}$,  denoted by $e_1$, $e_2$, \ldots, $e_{n}$, where $n-1$ is the genus of $X_0(p)$. They are in bijection with the isomorphism classes of supersingular elliptic curves $E_i/\overline{\F}$. The action of Hecke correspondences on $X$ induces an action on $\X$. 
There is a correspondence between modular forms of level $p$ and weight $2$ and elements of $\mathcal{X}\otimes \mathbb{C}$ that preserves the action
of the Hecke operators (\cite[Proposition ~5.6]{Gross}).
Let $v_E=\sum v_E(e_i)e_i\in \X$  be an eigenvector for all Hecke operators $t_m$ corresponding to $f_E$, i.e. $t_m v_E= a(m) v_E$, where $f_E(\tau)=\sum_{m=1}^\infty a(m)q^m$. We normalize $v_E$ (up to sign) such that the greatest common divisor of all its entries is $1$. 
 We define a $\Z$-bilinear pairing 
\[
\langle -,- \rangle:\X\times\X \rightarrow \Z,
\]
by requiring $\langle e_i,e_j \rangle=w_i \delta_{i,j}$ for all $i,j \in \{1,\ldots,n \}$, where
$w_i=\frac{1}{2}\#\textrm{Aut}(E_i)$.

We have the following key result of Mestre that relates the norm of $v_E$ to the modular degree $m_E$.

\begin{proposition} \cite[Theorem ~3]{Mestre} \label{prop:modular} 
$$\langle v_E,v_E \rangle= m_E  t,$$
where $t$ is the size of $E(\Q)_{tors}$.
\end{proposition}

The final ingredient we need is the Gross-Waldspurger formula on special values of $L$-series \cite{Gross}. An alternative approach is to use
 Gross-Kudla  formula for the special values of triple products of $L$-functions \cite{GK}.

In \cite{KK}, while studying supersingular zeros of divisor polynomials of elliptic curves, the authors posed the following conjecture.

\begin{conjecture}\label{conj:1} If $E$ is an elliptic curve of  prime conductor $p$, root number $1$, and $rank(E)>0$, then $v_E(e_i)$ is an even number for all $e_i$  with $j(E_i)\in \F.$ 
\end{conjecture}
The conclusion of the conjecture holds for any elliptic curve curve $E/\Q$ of 
prime conductor and root number $-1$, as well as for any curve of prime 
conductor that has positive discriminant and no rational points of order $2$ 
(see \cite[Theorems ~1.1,~1.2, ~1.4]{KK}).

In the last paragraph of this paper we will show the connection between this 
conjecture and Watkins' conjecture:

\begin{theorem}\label{thm:2}
Let $E/\Q$ be an elliptic curve of prime conductor such that $rank(E)>0$. If $v_E(e_i)$ is even number for all $e_i$  with $j(E_i)\in \F$, then $4|m_E$.
\end{theorem}

In particular, as remarked before, this verifies Watkins' conjecture if $E$ has prime conductor, $disc(E)>0$ and $rank(E)=2$.


\bigskip

\textbf{Acknowledgments:}  We would like to thank A. Dujella, I. Gusi\'c, M. Mereb and F. Najman for their useful comments and suggestions.

\section{Proof of the main Theorem}
We will give a series of propositions that will allow us to prove Theorem \ref{thm:1}.

\begin{proposition}\label{prop:GZ}
If $E/\Q$ has non-zero rank, then $L(E,1)=0$.
\end{proposition}
\begin{proof}
This is a classical application of Gross-Zagier and Kolyvagin theorems. For a 
reference see \cite[Theorem ~3.22]{Dar}.
\end{proof}

\begin{proposition} \label{prop:torsion}
If $E/\Q$ has prime conductor and non-zero rank, then $E(\Q)_{tors}$ is trivial.
\end{proposition}
\begin{proof}
This is a well known result; for example in  \cite{Mestre} it is shown that the isogeny classes of rational elliptic curves with conductor
 $p$ and non-trivial rational torsion subgroup are either $11.a$, $17.a$, $19.a$ 
and $37.b$, or the so called Neumann-Setzer curves that have a $2$-rational 
point. All these curves have rank $0$; this follows from a classical 
$2$-descent.
\end{proof}

\begin{proposition}\label{prop:0sum}
Let $v_E=\sum_{i=1}^n v_E(e_i) e_i\in \X$ be the vector corresponding to $f_E$. 
We have that $\sum_{i=1}^n v_E(e_i)=0$.
\end{proposition}
\begin{proof}
The vector $e_0=\sum_{i=1}^n \frac{e_i}{w_i}$ corresponds to the Eisenstein 
series. Moreover, the pairing $\langle -,- \rangle:\X\times\X \rightarrow \Z$ is 
compatible with the Hecke operators. Since the space of cuspforms is orthogonal 
to the Eisenstein series, we obtain  
\[ \left<v_E, e_0 \right>=\sum_{i=1}^n v_E(e_i)=0.  \]
\end{proof}

\begin{proposition}\label{prop:wi}
If $p\equiv 1, 5 \pmod{12}$, then all the $w_i$ are odd. On the other hand  if $p \equiv 7,11 \pmod{12}$
there is exactly one even $w_k$ (and in fact $w_k=2$).
\end{proposition}
\begin{proof}
See \cite[Table ~1.3  p.~117]{Gross}.
\end{proof}

Given $-D$ a fundamental negative discriminant, Gross defines
\[ b_D= \sum_{i=1}^{n} \frac{h_i(-D)}{u(-D)} e_i  ,\]
where $h_i(-D)$ is the number of optimal embeddings of the order of discriminant $-D$ into $End(E_i)$ modulo conjugation by $End(E_i)^{\times}$ and
$u(-D)$ is the number of units of the order.
We are in position to state (a special case of) Gross-Waldspurger formula 
\cite[Proposition ~13.5]{Gross}.
 
 \begin{proposition} \label{prop:Gross}
 If $-D$ is a fundamental negative discriminant with $ \left(\frac{-D}{p}\right)=-1$, then
 
 \[ L(E,1)L(E \otimes \varepsilon_D,1)= \frac{\left(f_E,f_E \right)}{\sqrt{D}} \frac{{m_D}^2}{ \langle v_E, v_E \rangle } ,\] 
 
 where $\varepsilon_D$ is the quadratic character associated to $-D$, $\left(f_E,f_E \right)$ is the Petersson inner product on $\Gamma_{0}(p)$ and

\[ m_D= \langle v_E, b_D \rangle .\]

\end{proposition}
 We will use the formula in the case that $-D=-4$ (and thus $p \equiv 3 \bmod{4}$). We know that $u(-4)=4$ and $h_{i}=0$ unless, following the notation
 from Proposition \ref{prop:wi}, $i=k$. In that case it is easy to see that $h_k(-4)=2$. Combining these observations we obtain that $b_4=\frac{1}{2}e_k$ where $k$ is the only index such that $w_k=2$ (this  corresponds to the elliptic curve $E_k$ with complex multiplication by $\Z[i]$).
  
Now we have the necessary ingredients in order to prove Theorem \ref{thm:1}.

\begin{proof}[Proof of Theorem \ref{thm:1}]
As remarked in the introduction, it is enough to prove the theorem when $E$ has prime conductor $p$ and it is
the strong Weil curve.
Suppose on the contrary that $E$ has positive rank. In consequence, by Proposition \ref{prop:modular}
and Proposition \ref{prop:torsion}  we know that $\langle v_E, v_E \rangle$ must be odd.
Moreover,  
\[ \left<v_E, v_E\right>=\sum_{i=1}^n w_i v_E(e_i)^2\equiv \sum_{i=1}^n w_i v_E(e_i)  \pmod{2}. \]
 Using Propositions  \ref{prop:0sum} and \ref{prop:wi} we obtain that if 
$p\equiv 1, 5 \pmod{12}$  $\left<v_E, v_E\right>$ is even  and if $p \equiv 7,11 
\pmod{12}$ then $\left<v_E, v_E\right> \equiv v_E(e_k) \pmod{2}$, where $k\in 
\N$ is the only index such that $w_k=2$. In that case, since $L(E,1)=0$ (by 
Proposition \ref{prop:GZ}),
Proposition \ref{prop:Gross} implies that 

\[ m_4=\left<v_E, b_4 \right>=0  .\]
 Since $b_4=\frac{1}{2}e_k$, we get that
\[ m_4=v_E(e_k)=0.\]
 Therefore, $\left<v_E, v_E\right>$ is even, leading to a contradiction.
\end{proof}
\begin{remark}
 Another proof along the same lines uses that if $L(E,1)=0$ then 
 \[\sum_i {w_i}^ 2 v_E(e_i)^3=0 .\] 
 This is proved in \cite[Corollary ~11.5]{GK}, as a 
consequence of the Gross-Kudla  formula of special values of triple product 
$L$-functions.
  The number $\sum_i {w_i}^ 2 v_E(e_i)^3$  clearly has the same parity as $\langle v_E, v_E \rangle$, leading to the desired contradiction.
 
\end{remark}

\section{The proof of the Theorem \ref{thm:2}}
\begin{proof}[Proof of Theorem \ref{thm:2}] For a given $e_i$, denote by 
$\bar{i}\in \{1, 2, \ldots, n\}$ the unique index such that $e_{\bar{i}}$ 
corresponds to the curve $E_i^p$. Then 
\cite[Proposition ~2.4]{Gross} implies that $v(e_i)=v(e_{\bar{i}})$. As in the 
proof of Theorem \ref{thm:1}, we have that $v_E(e_k)=0$ whenever $w_k \ne 1$, 
hence Proposition \ref{prop:torsion} implies that 
\[m_E \equiv \sum_i v_E(e_i)^2 \pmod{4}.  \]

If $E_i$ is defined over $\F$ (i.e. $\bar{i}=i$), then by the assumption \[ v_E(e_i)^2 \equiv 0 \pmod{4}.\]
 Hence \[m_E\equiv \sum_{i}^{'} 2 v_E(e_i)^2 \pmod{4},\] where we sum over the pairs $\{i, \bar{i}\}$ with $i \ne \bar{i}$. Note that Gross-Kudla formula implies that \[\sum_i v_E(e_i)^3 \equiv \sum_{i}^{'} 2 v_E(e_i)\equiv 0 \pmod{4},\] where the second sum is over the pairs $\{i, \bar{i}\}$ for which $v_E(e_i)$ is odd. It follows that the number of such pairs is even, hence $m_E\equiv 0 \pmod{4}$.

\end{proof}

\bibliographystyle{alpha}

\end{document}